\documentclass{amsart}

\usepackage{amssymb, hyperref}   
\usepackage{amsthm, amsmath,amssymb,amsbsy,amsfonts, latexsym,amsopn,amstext, amsxtra,euscript,amscd}
\usepackage{enumerate}
\usepackage{multirow}

\hypersetup{
  colorlinks   = true, 
  urlcolor     = blue, 
  linkcolor    = blue, 
  citecolor   = red 
}

\usepackage[msc-links]{amsrefs}

\newtheorem{thm}{Theorem}

\newtheorem {prop}{Proposition}
\newtheorem {lem}{Lemma}
\newtheorem {rem}{Remark}
\newtheorem {cor}{Corollary}
\theoremstyle{remark}

\DeclareMathOperator\chara{char }

\DeclareMathOperator\bAut{\overline{Aut}\, }
\DeclareMathOperator\Aut{Aut \, }

\newcommand\C{\mathbb C}

\def\H{\mathcal H}
\newcommand\M{{\mathcal M}}

\newcommand\X{\mathcal X}            

\newcommand{\s}{}

\def\>{\rangle}
\def\<{\langle}
\def\bG{\overline G}
  
\def\p{\mathfrak p}
\def\iso{\cong }

\begin{document}

\title{On the field of moduli of superelliptic curves}

\author{Ruben Hidalgo}
\address{Departamento de Matem\'atica y Estad\'{\i}stica, Universidad de La Frontera, Temuco, Chile.}
\email{ruben.hidalgo@ufrontera.cl}
\thanks{The first author was partially supported by Project Fondecyt 1150003 and Project Anillo ACT 1415 PIA-CONICYT}

\author{Tony Shaska}
\address{Department of Mathematics and Statistics, Oakland University, Rochester, MI, 48386. }
\email{shaska@oakland.edu}

\begin{abstract}

A superelliptic curve $\X$ of genus $g\geq 2$ is not necessarily defined over its field of moduli but it can be defined over a quadratic extension of it. While a lot of work has been done by many authors to determine which  hyperelliptic curves are defined over their field of moduli, less is known for superelliptic curves. 

In this paper we observe that if the reduced group of a genus $g\geq 2$ superelliptic curve $\X$ is different from the trivial or cyclic group, then $\X$ can be defined over its field of moduli; in the cyclic situation we provide a sufficient condition for this to happen. We also determine those families of superelliptic curves of genus at most $10$ which might not be definable over their field of moduli.
\end{abstract}


\maketitle


\def\e{\varepsilon}

\newcommand\G{\bar{G}}
\newcommand\normal{\triangleleft }

\def\bC{\bold C}
\def\a{\alpha}
\def\b{\beta}
\def\g{\gamma}
\def\d{\delta} 


\section{Introduction}
Let $k$ be an algebraically closed field of  characteristic zero and $\X$ a genus $g\geq 2$, projective, irreducible algebraic curve defined over $k$. The {\it field of moduli} of $\X$ is the field of definition of the corresponding moduli point $\p=[\X]$  in the moduli space $\M_g$ (see Section \ref{Sec:FOD} for a formal definition of the field of moduli). In general, to determine the field of moduli and to decide if it is a field of definition is  difficult task and it is an active research topic.
Examples of algebraic curves (for $k={\mathbb C}$) which cannot be defined over their field of moduli have been provided  
by Earle \cite{Earle}, Huggins \cite{Hu2} and Shimura \cite{Shimura} for the hyperelliptic situation and by the first author \cite{Hid} and Kontogeorgis \cite{Kontogeorgis} in the non-hyperelliptic situation. 
In other words, $\M_g$ is not a \textit{fine} moduli space.

By results due to Koizumi \cite{Koizumi}, there is always a field of definition that is a finite extension of the field of moduli and if, moreover, the field $k$ has transcendental numbers over the field of moduli (for instance $k={\mathbb C}$), then the field of moduli coincides with the intersection of all the fields of definitions of the curve. On the other hand, if $k=\overline{\mathbb Q}$, then it might be that the intersection of all fields of definitions (inside $k$) is different from the field of moduli.

Investigating the obstruction for the field of moduli to be a field of definition is part of descent theory for fields of definition and has many consequences in arithmetic geometry. Many works have been devoted to this problem, most notably by Weil \cite{Weil}, Shimura \cite{Shimura} and Grothendieck,  among many others.  Weil's criterion \cite{Weil} assures that if a curve has no non-trivial automorphisms then its field
of moduli is a field of definition. On the other extreme, if the curve $\X$ is quasiplatonic (that is, when the quotient orbifold $\X/{\rm Aut}(\X)$ has genus zero and exactly three cone points), then Wolfart \cite{Wolfart} proved that the field of moduli is also a field of definition. Hence, the real problem occurs when the curve has non-trivial automorphism group but the quotient orbifold $\X/{\rm Aut}(\X)$ has non-trivial moduli.

Mestre \cite{Me} provided the first algorithm which determines if the field of moduli is a field of definition for genus two curves with automorphism group of order $2$.   The algorithm gives explicitly an equation of the curve over its field of moduli when such equation exists.  In \cite{ants} it is shown that the field of moduli is a field of definition for genus two curves with automorphism group isomorphic to the dihedral groups $D_4$ or $D_6$.    Cardona and Quer  \cite{Ca} shows that this is true also for genus two curves with automorphism group the Klein $4$-group $D_{2}=V_{4}$.  Most recently Malmendier and Shaska \cite{Ma-sha} generalized the method of Mestre and are able to construct a universal genus two curve defined over a quadratic number field. 
In all these papers the moduli point $\p \in \M_2$  was computed and an explicit equation  of a curve was determined via the invariants of the binary sextics.  

Shaska conjectured in   \cite{conj} and \cite{dihedral}   that  all hyperelliptic curves $\X$ with extra automorphisms (i.e. $|\Aut (\X) | > 2$)  are defined over their field of moduli. This was proved to be incorrect by Huggins \cite{Hu2}  who showed explicit examples of hyperelliptic curves with reduced automorphism group isomorphic to a cyclic group which cannot be defined over their field of moduli (curves defined over $k={\mathbb C}$ which cannot be defined over ${\mathbb R}$ and whose field of moduli is contained inside ${\mathbb R}$).  In the same paper Huggins proved that every hyperelliptic curve whose reduced automorphism group is different from a cyclic group (including the trivial situation) is definable over its field of moduli; even they are hyperelliptically defined over it. Kontogeorgis \cite{Kontogeorgis} generalized the above result to cyclic $p$-gonal curves (where $p$ is a prime integer) and in \cite{HQ} the first author and Quispe generalized the above for curves admitting a subgroup of automorphisms being unique up to conjugation (note that in case of cyclic $p$-gonal curves, in the case that it is definable over its field of moduli, a rational model over its field of moduli may not be in a cyclic $p$-gonal form as was the case of the hyperelliptic situation).

A genus $g \geq 2$ smooth algebraic curve (or closed Riemann surface)   $\X$ is called {\it superelliptic of level $n$}  if there exist an element $\tau \in \Aut( \X)$ of order $n$ such that (i) $\tau$ is central, (ii) the quotient $\X / \< \tau \> $ has genus zero,  (iii) every cone point has order $n$ and (iv) the rotation number of $\tau$ at each of its fixed points is the same (these are very special types of cyclic $n$-gonal curves); in which case $\tau$ is called  a {\it superelliptic automorphism of level $n$} and $H=\langle \tau \rangle$ a {\it superelliptic group of level $n$}.  In Proposition \ref{unicidad} we observe that the superelliptic group of level $n$ is unique.

These superelliptic Riemann surfaces are natural generalizations of hyperelliptic Riemann surfaces (in which case $n=2$ and $\tau$ is given by the hyperelliptic involution), but they are in general non-hyperelliptic ones. Superelliptic curves are the natural cases where one could try to extend the results from hyperelliptic curves; see \cite{b-zhupa}. These are, with some exceptions, the only classes of curves for which we can easily write down equations starting from the automorphism group of the curve.  These curves have equations of the form
$$\X: \quad y^{n}=\prod_{i=1}^{r} (x-a_{i})$$
so that either $r \equiv 0 \mod(n)$ or $\gcd(n,r)=1$, and we may assume that the generator of the superelliptic group of level $n$ is generated by 
 $\tau:(x,y) \mapsto (x,\omega y)$, where $\omega^{n}=1$ (recall that we are assuming $\tau$ to have the same rotation number about all of its fixed points, so we may choose a power of it, of the same order $n$, for which the rotation number is equal to one).

Beshaj and Thompson \cite{b-th} have proved that a superelliptic curve can always be defined over a quadratic extension of its field of moduli. Moreover, they have provided an equation over an at most quadratic extension of its field of moduli using the Shaska invariants which has the above particular cyclic gonal form. At this point, we should remark that if a superelliptic curve $S$ is definable over its field of moduli, then it might be that it is not definable in a cyclic gonal form over it; but it will be so definable over a quadratic extension of it. It has been observed by the second author that the equation provided in \cite{b-th} fails to give a genus $g\geq 2$ superelliptic curve exactly when the curve is quasiplatonic.

In general, the field of moduli of a superelliptic curve $\X$ is not a field of definition, but we will see (cf. Thm.~\ref{thm1}) that 
if $| \Aut (\X) | > n$, where $n$ is the order of $\tau$,  for the majority of  cases,  the field of moduli is a field of definition.  

A superelliptic descent over a field $k$ means that we may find a defining curve with equation $y^{n}=f(x)$ where $f \in k[x]$, which is a slightly more restricted class than the typical descent that ask for a curve over $k$ not necessarily of the above form; see for example \cite{lrs} for the typical descent for hyperelliptic situation for $g=3$ (in there there is an example of a hyperelliptic curve of genus three, which is definable over ${\mathbb Q}$, but it cannot be hyperelliptically defined over ${\mathbb Q}$).

The main goal of this paper is to (partially) describe which superelliptic curves of genus $g \leq 10$ are definable over their field of moduli. In each genus (see Tables 1,2 and 3), we describe those ones which might not be definable over their field of moduli;  all cases which are not in blue are defined over their field of moduli and the blue ones might or not be definable over the field of moduli.

The results of this paper can be very easily extended to positive characteristic in the case when the covering $\X \to \X/\< \tau \>$ is tame.  The list of full automorphisms groups in positive characteristic ($\chara k = p >0$)  was determined in \cite{Sa} and the rest of the proofs should follow easily when $(n, p) =1$.  We only focus in characteristic zero.

The paper is organized as follows.
In Section \ref{Sec:prelim} we give a brief review of superelliptic curves and their automorphism groups. Such groups  were fully classified in \cite{Sa} and the equations of each parametric curve for any given group are given in \cite{Sa-sh} for any given genus $g \geq 2$; in Section \ref{Sec:tablas} we provide such a list for 
$5 \leq g \leq 10$ (see also \cite{super1} and \cite{zhupa}). 
In Section \ref{Sec:FOMSC} we recall some known results which provide sufficient conditions for a curve to be definable over its field of moduli. Using such conditions, 
we give the list of all possible superelliptic curves of genus $2 \leq g \leq 10$ which might not be definable over the field of moduli (in particular, the complementary ones are definable over them). Let us note that 
the case of genus $g=2, 3$ are already well known and $g=4$ can be obtained by \cite{b-zhupa}. 

\medskip

\noindent \textbf{Notation:}  
Throughout this paper,  $k$ denotes an algebraically closed field of characteristic zero  and $\X$ a genus $g\geq 2$, smooth, projective, irreducible, algebraic curve defined over $k$. We will keep the notation used in \cite{Sa} to denote certain groups (i.e, $C_n$ is the cyclic group of order $n$, $D_{m}$ the dihedral group of order $m$, $V_{4}=D_{4}$ the $4$-Klein group, etc).

\section{Preliminaries}\label{Sec:prelim}
In this section, $k$ will be a fixed algebraically closed field of characteristic zero and we denote by ${\rm Gal}(k)$ its group of field automorphisms.

\subsection{The field of moduli and fields of definition}
Let $\X$ be a genus $g$ projective, irreducible, algebraic curve defined over $k$, say given as the common zeroes of the polynomials $P_{1},\ldots, P_{r}$, and let us denote by $G=\Aut (\X)$ the full automorphism group of $\X$.  

If $\sigma \in {\rm Gal}(k)$, then $X^{\sigma}$ will denote the curve defined as the common zeroes of the polynomials $P_{1}^{\sigma},\ldots,P_{r}^{\sigma}$, where $P_{j}^{\sigma}$ is obtained from $P_{j}$ by applying $\sigma$ to its coefficients. In particular, if $\tau$ is also a field automorphism of $k$, then $X^{\tau \sigma}=(X^{\sigma})^{\tau}$.

\subsubsection{Field of definition}
A subfield $k_{0}$ of $k$ is called a {\it field of definition} of $\X$ if there is a curve ${\mathcal Y}$, defined over $k_{0}$, which is isomorphic to $\X$. It is clear that every subfield of $k$ containing $k_{0}$ is also a field of definition of it. In the other direction, a subfield of $k_{0}$ might not be a field of definition of $\X$.

Weil's descent theorem \cite{Weil} provides sufficient conditions for a subfield $k_{0}$ of $k$ to be a field of definition. Let us denote by ${\rm Gal}(k/k_{0})$ the group of field automorphisms of $k$ acting as the identity on $k_{0}$.

\s

\begin{thm}[Weil's descent theorem \cite{Weil}]
Assume that for every $\sigma \in {\rm Gal}(k/k_{0})$ there is an isomorphism $f_{\sigma}:\X \to \X^{\sigma}$ so that
$$f_{\tau\sigma}=f_{\sigma}^{\tau} \circ f_{\tau}, \quad \forall \sigma, \tau \in {\rm Gal}(k/k_{0}).$$

Then there is a curve  ${\mathcal Y}$, defined over $k_{0}$, and there is an isomorphism $R:\X \to {\mathcal Y}$, defined over a finite extension of $k_{0}$, so that $R=R^{\sigma} \circ f_{\sigma}$, for every $\sigma \in {\rm Gal}(k/k_{0})$.
\end{thm}

\s

Clearly, the sufficient conditions in Weil's descent theorem are trivially satisfied if $\X$ has non-trivial automorphisms (a generic situation for $\X$ of genus at least three).

\s

\begin{cor}\label{coro:weil}
If $\X$ has trivial group of automorphisms and for every $\sigma \in {\rm Gal}(k/k_{0})$ there is an isomorphism $f_{\sigma}:\X \to \X^{\sigma}$, then $\X$ can be defined over $k_{0}$.
\end{cor}

\s

\subsubsection{Field of moduli} \label{Sec:FOD}
The notion of field of moduli was originally introduced by Shimura for the case of abelian varieties and later extended to more general algebraic varieties by Koizumi.
If $G_{\X}$ is the subgroup of ${\rm Gal}(k)$ consisting of those $\sigma$ so that $\X^{\sigma}$ is isomorphic to $\X$, then the fixed field $M_{\X}$ of $G_{\X}$ is called {\it the field of moduli} of $\X$.

As we are assuming that $k$ is algebraically closed and of characteristic zero, we have that $G_{\X}$ consists of all automorphisms of ${\rm Gal}(k)$ acting as the identity on $M_{\X}$.

It is known that every curve of genus $g \leq 1$ can be defined over its field of moduli. If $g \geq 2$, then (as already said in the Introduction) there are known examples of curves which cannot be defined over their field of moduli.  A direct consequence of Corollary \ref{coro:weil} is the following.

\s

\begin{cor}\label{corotrivial}
Every curve with trivial group of automorphisms can be defined over its field of moduli.
\end{cor}

\s

As a consequence of Belyi's theorem \cite{Belyi}, every quasiplatonic curve $\X$ can be defined over $\overline{\mathbb Q}$ (so over a finite extension of ${\mathbb Q}$).

\s

\begin{thm}[Wolfart \cite{Wolfart}]\label{Wolfart}
Every quasiplatonic curve can be defined over its field of moduli (which is a number field).
\end{thm}

\s

\subsection{Two practical sufficient conditions}
When the curve $\X$ has a non-trivial group of automorphisms, then Weil's conditions (in Weil's descent theorem) are in general not easy to check. Next we consider certain cases for which it is possible to check for $\X$ to be definable over its field of moduli.

\subsubsection{Sufficient condition 1: unique subgroups}
Let $H$ be a subgroup of $\Aut(\X)$. In general it might be another different subgroup $K$ which is isomorphic to $H$ and with $\X/K$ and $\X/H$ having the same signature. For instance, the genus two curve $\X$ defined by $y^{2}=x(x-1/2)(x-2)(x-1/3)(x-3)$ has two conformal involutions, $\tau_{1}$ and $\tau_{2}$, whose product is the hyperelliptic involution. The quotient $\X/\langle \tau_{j}\rangle$ has genus one and exactly two cone points (of order two). 

We say that $H$ is 
is {\it unique} in $\Aut(\X)$ if it is the unique subgroup of $\Aut(\X)$ isomorphic to $H$ and with quotient orbifold of same signature as  $\X/H$. Typical examples are (i) $H=\Aut(\X)$ and (ii) $H$ being the cyclic group generated by the hyperelliptic involution for the case of hyperelliptic curves. 

If $H$ is unique in $\Aut(\X)$, then it is a normal subgroup; so we may consider the reduced group $\bAut(\X)=\Aut(\X)/H$, which is a group of automorphisms of the quotient orbifold $\X/H$. In \cite{HQ} the following sufficient condition for a curve to definable over its field of moduli was obtained.

\s

\begin{thm}[Hidalgo and Quispe \cite{HQ}]\label{thm1}
Let $\X$ be a curve of genus $g \geq 2$ admitting a subgroup $H$ which is unique in $\Aut(\X)$ and so that $\X/H$ has genus zero.  If the reduced group of automorphisms $\bAut(\X)=\Aut(\X)/H$ is different from trivial or cyclic, then $\X$ is definable over its field of moduli.
\end{thm}

\s

If $\X$ is a hyperelliptic curve, then a consequence of the above is the following result (originally due to Huggins \cite{Hu2}).

\s

\begin{cor}\label{cor1}
Let $\X$ be a hyperelliptic curve with extra automorphisms and reduced automorphism group $\bAut (\X)$ not isomorphic to a cyclic group.  Then, the field of moduli of $\X$  is a field of definition. 
\end{cor}

\s

\subsubsection{Sufficient condition 2: Odd signature}
Another sufficient condition of a curve $\X$ to be definable over its field of moduli, which in particular contains the case of quasiplatonic curves, was provided in \cite{AQ}. We say that $\X$ has {\it odd signature} if $\X/{\rm Aut}(\X)$ has genus zero and in its signature one of the cone orders appears an odd number of times.

\s

\begin{thm}[Artebani and Quispe \cite{AQ}\label{thm2}]
Let $\X$ be a curve of genus $g \geq 2$. If $\X$ has odd signature, then it can be defined over its field of moduli.
\end{thm}

\s

\subsection{The locus of curves with prescribed group action, moduli dimension of families}
Fix an integer $g\ge2$ and a finite group $G$. Let $C_1, \dots , C_r$ be nontrivial conjugacy classes of $G$. Let $\bC=(C_1, \dots ,C_r)$, viewed as an unordered tuple, where repetitions are allowed. We allow $r$ to be zero, in which case $\bC$ is empty. Consider pairs $(\X, \mu)$, where $\X$ is a curve and $\mu: G  \to  \Aut(\X)$ is an injective
homomorphism. We will suppress $\mu$ and just say $\X$ is a curve with $G$-action, or a $G$-curve. Two $G$-curves $\X$ and $\X'$ are called equivalent if there is a $G$-equivariant  conformal isomorphism $\X\to \X'$.

We say a $G$-curve $\X$ is \textbf{of ramification type} $(g, G, \bC)$ (for short, of type $(g,G,\bC)$)
if 
\begin{enumerate}
\item[i)] $g$ is the genus of $\X$, 
\item[ii)] $G<{\rm Aut}(\X)$,
\item[iii)]  the points of the quotient $\X/G$ that are ramified in the cover $\X\to \X/G$ can be labeled as $p_1, \dots ,p_r$ such that $C_i$ is the conjugacy class in $G$ of distinguished inertia group
generators over $p_i$ (for $i=1, \dots ,r$). 
\end{enumerate}

If $\X$ is a $G$-curve of type $(g,G,\bC)$, then the genus $g_0$ of $\X/G$ is given by the Riemann-Hurwitz formula
$$2(g-1)=2|G|(g_{0}-1)+|G|\sum_{j=1}^{r}(1-|C_{j}|^{-1}).$$

Define $\H=\H(g,G,\bC)$ to be the set of equivalence classes of $G$-curves of type $(g,G,\bC)$. By covering space theory, $\H$ is non-empty if and only if $G$ can be generated by elements $\a_1,\b_1,...,\a_{g_0},\b_{g_0},\g_1,...,\g_r$ with $\g_i\in C_i$ and $\prod_{j}\ [\a_j,\b_j]\  \prod_{i} \g_i =  1$, where $[\a,\b]=\ \a^{-1}\b^{-1}\a\b$.

Let $\M_g$ be the moduli space of genus $g$ curves, and $\M_{g_0,r}$ the moduli space of genus $g_0$ curves with $r$ distinct marked points, where we view the marked points as unordered. Consider the map 
\[ \Phi:\ \H\ \to \ \M_{g},\]
forgetting the $G$-action, and the map $\Psi:\ \H \ \to \ \M_{g_0,r} $ mapping (the class of) a $G$-curve $\X$ to the class of the quotient curve $\X/G$ together with the (unordered) set of branch points $p_1, \dots , p_r$. 

If $\H \ne \emptyset$, then $\Psi$ is surjective and has finite fibers, by covering space theory. Also $\Phi$ has finite fibers, since the automorphism group of a
curve of genus $\ge2$ is finite. The set $\H$ carries a structure of quasi-projective variety (over $\C$) such that the maps $\Phi$ and $\Psi$ are finite morphisms. 
If $\H\ne\emptyset$, then all (irreducible) components of $\H$ map surjectively to $\M_{g_0,r}$ (through a finite map), hence they all have the same dimension
\[ \d     (g,G,\bC):= \ \ \dim\ \M_{g_0,r} \ \ = \ \ 3g_0-3+r\]  
Let $\M(g,G,\bC)$ denote the image of $\Phi$, i.e., the locus of genus $g$ curves admitting a $G$-action of type $(g,G,\bC)$. 
 Since $\Phi$ is a finite map, if this locus is non-empty, each of its (irreducible) components has dimension $\d (g, G,\bC)$.

Theorem \ref{teoAQ} can be written as follows.

\s

\begin{thm}
If $\d (g, G,\bC) = 0$, then every curve in $\M (g, G, \bC) $ is defined over its field of moduli.
\end{thm}

\s

The last part of the above is due to the fact that $\delta=0$ ensures that the quotient orbifold $\X/G$ must be of genus zero and with exactly three conical points, that is, $\X$ is a quasiplatonic curve.

\section{Field of moduli of superelliptic curves}\label{Sec:FOMSC}
\subsection{Automorphism groups of superelliptic curves}
Let $\X$ be a superelliptic curve of level $n$ with $G={\rm Aut}(\X)$. By the definition, there is some $\tau \in G$, of order $n$ and central, so that the quotient $\X / \< \tau \>$ has genus zero, that is, it can be identified with the projective line, and all its cone points have order $n$.  As, in this case, the cyclic group $H=\< \tau \> \cong C_{n}$  is normal subgroup of $G$, we may consider the quotient group $\G \, := \, G/H$, called the \textit{reduced automorphism group of $\X$ with respect to $H$}; so  $G$ is a degree $n$ central extension of $\G$.

\s

In the particular case that $n=p$ is a prime integer, Castelnuovo-Severi's inequality \cite{CS} asserts that for $g>(p-1)^{2}$ the cyclic group $H$ is unique in $\Aut(\X)$. 
In \cite{Hid-pgrupo} it is observed that if $n=p \geq 5(r-1)$ is prime, where $r \geq 3$ denotes the number of cone points of the quotient $\X/H$, then $H$ is again unique. The following result shows that the superelliptic group of level $n$ is unique.

\s

\begin{prop}\label{unicidad}
A superelliptic curve of level $n$ and genus $g \geq 2$ has a unique superelliptic group of level $n$.
\end{prop}

\begin{proof}
Let $\X$ be a superelliptic curve of level $n$ and assume that $\langle \tau \rangle$ and $\langle \eta \rangle$ are two different superelliptic groups of level $n$.
The condition that the cone points of both quotient orbifolds $\X/\langle \tau \rangle$ and $\X/\langle \eta \rangle$ are of order $n$ asserts that a fixed point of a non-trivial power of $\tau$ (respectively, of $\eta$) must also be a fixed point of $\tau$ (respectively, $\eta$). In this way, our previous assumption asserts that no non-trivial power of $\eta$ has a common fixed point with a non-trivial power of $\tau$. In this case, the fact that $\tau$ and $\eta$ are central asserts that $\eta \tau=\tau \eta$ and that $\langle \tau, \eta \rangle \cong C_{n}^{2}$ (see also \cite{Sa}). Let $\pi:\X \to {\mathbb P}^{1}_{k}$ be a regular branched cover with $\langle \tau \rangle$ as deck group. Then the automorphism $\eta$ induces a automorphism $\rho \in {\rm PGL}_{2}(k)$ (also of order $n$) so that $\pi \eta = \rho \pi$. As $\rho$ is conjugated to a rotation $x \mapsto \omega_{n} x$, where $\omega_{n}^{n}=1$, we observe that it has exactly two fixed points. This asserts that $\eta$ must have either $n$ or $2n$ fixed points (forming two orbits under the action of $\langle \tau \rangle$). As this is also true by interchanging the roles of $\tau$ and $\eta$, the same holds for the fixed points of $\tau$. It follows that the cone points of $\pi$ consists of (i) exactly two sets of cardinality $n$ each one or (ii) exactly one set of cardinality $n$, and each one being invariant under the rotation $\rho$. Up to post-composition by a suitable transformation in ${\rm PGL}_{2}(k)$, we may assume these in case (i) the $2n$ cone points are given by the $n$ roots of unity and the $n$ roots of unity of a point different from $1$ and $0$ and in case (ii) that the $n$ cone points are the $n$ roots of unity. In other words, $\X$ can be given either as
$$\X_{1}: \; y^{n}=(x^{n}-1)(x^{n}-a^{n}), \quad a \in k-\{0,1\}$$
or as the classical Fermat curve
$$\X_{2}: \;y^{n}=x^{n}-1$$
and, in these models, 
\[ \tau(x,y)=(x,\omega_{n} y), \quad \eta(x,y)=(\omega_{n} x, y).\]

As the genus of $\X_{1}$ is at least two, we must have that $n \geq 3$. But such a curve also admits the order two automorphism
\[ \gamma(x,y)=  \left(\frac{a}{x},\frac{ay}{x^{2}}  \right)\]
which does not commute with $\eta$, a contradiction to the fact that $\eta$ was assumed to be central.

In the Fermat case, the full group of automorphisms is $C_{n}^{2} \rtimes S_{3}$ and it may be checked that it is not superelliptic.

\end{proof}

\s
\subsection{Most of superelliptic curves are definable over their field of moduli}
The group  $\G$ is a subgroup of the group of automorphisms of a genus zero field, so $\G <  PGL_2(k)$ and $\G$ is finite. It is a classical result that every finite subgroup of $PGL_2 (k)$ (since we are assuming $k$ of characteristic zero) is either the trivial group or isomorphic to one of the following: $C_m $, $ D_m $, $A_4$, $S_4$, $A_5$. All automorphisms groups of superelliptic curves and their equations were determined in \cite{Sa} and \cite{Sa-sh}. 
Determining the automorphism groups $G$, the signature $\bC$ of the covering $\X \to \X/G$, and the dimension of the locus $\M(g,G,\bC)$  for superelliptic curves is known (see, for instance, \cite{Sa}). 

We have seen in Theorem \ref{unicidad} that its superelliptic group of level $n$ is unique.
As a consequence of Theorem \ref{thm1}, we obtain the following fact concerning the field of moduli of superelliptic curves. 

\begin{thm}\label{teounico}
Let $\X$ be a superelliptic curve of genus $g \geq 2$ with superelliptic group $H \cong C_{n}$. If the reduced group of automorphisms $\bAut(\X)=\Aut(\X)/H$ is different from trivial or cyclic, then $\X$ is definable over its field of moduli.
\end{thm}

\s

As a consequence of the above, we only need to take care of the case when the reduced group $\G=G/H$ is either trivial or cyclic. As a consequence of Theorem \ref{thm2} we have the following fact.

\begin{thm}\label{teoAQ}
Let $\X$ be a superelliptic curve of genus $g \geq 2$ with superelliptic group $H \cong C_{n}$ so that $\G=G/H$ is either trivial or cyclic. If $\X$ has odd signature, then it can be defined over its field of moduli.
\end{thm}

\s

As a consequence, the only cases we need to take care are those superelliptic curves with reduced group $\G=G/H$ being either trivial or cyclic and with $\X/G$ having not an odd signature.

\section{Superelliptic curves of genus at most 10}
Using the list provided in Section \ref{Sec:tablas} and the previous results, we proceed, in each genus $2 \leq g \leq 10$, to describe those superelliptic curves which are definable over their field of moduli. Observe that in the left cases (which might or might not  be definable over their field of moduli) the last column provides an algebraic model $y^{n}=f(x)$, where $f(x)$ is defined over the algebraic closure and not necessarily over a minimal field of definition. The branched regular covering $\pi:\X \to {\mathbb P}_{k}^{1}$ defined by $\pi(x,y)=x$ as deck group $H=\langle \tau(x,y)=(x,\epsilon_{n} y) \rangle \cong C_{n}$. 

\s
\subsubsection{\bf Genus $2$}
The case of genus $g=2$ is well known since in this case every curve $\X$ such that $|\Aut (\X) | > 2$ the field of moduli is a field of definition. There are examples of genus two curves, whose reduced group is trivial, which are not definable over their field of moduli.

\s
\subsubsection{\bf Genus $3$}
There are 21 signatures for genus $g=3$ from which 12 of them are hyperelliptic and $3$ are trigonal. 

\begin{lem}
Every superelliptic curve of genus $3$, other than Nr. 1 and 2 in Table~\ref{g=3}, is definable over its field of moduli. 
\end{lem} 
 
\begin{proof}
 If $\bAut (\X)$ is isomorphic to $A_4$ or $S_4$ then the corresponding locus consists of the curves $y^4= x^4+ 2x^2 + \frac 1 3 $ and $y^2=   x^8 + 14 x^4 + 1$ which are both defined over their field of moduli. 

If $\bAut (\X)$ is isomorphic to a dihedral group and $\X$ is not hyperelliptic, then $\Aut (\X)$ is isomorphic to $V_4 \times C_4$, $G_5$, $D_6 \times C_3$, and $G_8$.  These cases  $G_5$, $D_6 \times C_3$, and $G_8$ correspond to $y^4=  x^4-1$, $y^3= x (x^3-1)$, and $y^4 = x (x^2-1)$, which are all defined over the field of moduli.  

\begin{table}[h]

\caption{Genus $3$ curves No. 1 and 2 are the only one whose field of moduli is not necessarily a field of definition}

\begin{tabular}{|l|l|l|l|l|l|l|l|}
\hline
Nr. & $\overline G$             & G      &$n$  &$m$ & sig. & $\delta$ & Equation $y^n=f(x)$ \\
\hline \hline
\textbf{\color{blue}\color{blue}1} &$\{I\}$ & $C_2$           &  2 & 1 & $2^8$  & 5 &  $  x \left( x^6+ \sum_{i=1}^5 a_i x^i   +1 \right) $ \\
\textbf{\color{blue}\color{blue}2} & $C_2$ &  $ V_4$         &  2 & 2 & $2^6$       & 3 &  $ x^8+ a_1 x^2+ a_2 x^4+ a_3 x^6+1 $ \\ 
3 &$C_2$ &         $C_4$  &  2 & 2 & $2^3, 4^2 $         & 2 &   $ x \left(x^6+a_1 x^2+a_2 x^4+1 \right)$ \\
4 &$C_2$ &    $C_6$       &  3 & 2 & $2, 3^2,6$             & 1 &   $  x^4+a_1 x^2+1 $ \\
5 & $V_{4}$ & $V_4\times C_4$ &  4 & 2 & $2^3, 4$         & 1 &  $ x^4+a_1 x^2+1$ \\
\hline \hline
\end{tabular}
\label{g=3}
\end{table}

If $\bAut(\X)$ is isomorphic to a cyclic group, then in the cases when it is isomorphic to $C_{14}, C_{12}$ there are two cases which correspond to the curves $y^2= x^7+1$ and $y^3= x^4+1$. The left cases  are given in Table~\ref{g=3}. The curve No. 5 is definable over its field of moduli by Theorem \ref{teounico}. All the other cases, with the exception of Nr. 1 and 2,  the curves are of odd signature, so they are definable over their field of moduli by Theorem \ref{teoAQ}.

\end{proof}

\subsubsection{\bf Genus $4$}
\begin{lem}
Every superelliptic curve of genus 4, other than Nr. 1, 3 and 5 in Table~\ref{g=4}, is definable over its field of moduli.
\end{lem}

\begin{proof}

There is only one case when  the reduced automorphism group $\bAut(\X)$  is not isomorphic to a cyclic or a dihedral group, namely $\bG\iso S_4$.  In this case, the curve is $y^3 = x (x^4-1)$ and is defined over the field of moduli. If $\bG$ is isomorphic to a dihedral group, then there are only $6$ signatures which give the groups $D_6\times C_3$, $D_4\times C_3$, $D_{12}\times C_3$, $D_4\times C_3$, $D_8 \times C_3$, and $D_4 \times C_5$.  The groups $D_{12}\times C_3$, $D_8 \times C_3$, and $D_4 \times C_5$ correspond to curves $y^3 = x^6-1$, $y^3=x(x^4-1)$, and $y^5= x(x^2-1)$ respectively.  
The remaining three cases are given by Nrs. 7, 8 and 9 in Table~\ref{g=4} which are definable over their field of moduli by Theorem \ref{teounico}.

\begin{table}[hb]

\caption{Genus $4$ curves No. 1, 3 and 5 are the only ones whose field of moduli is not necessarily a field of definition}

\begin{tabular}{|l|l|l|l|l|l|l|l|}
\hline
Nr. & $\overline G$                & G      &$n$  &$m$ & sig. & $\delta$ & Equation $y^n=f(x)$ \\
\hline \hline
\textbf{\color{blue}\color{blue}1} & & $C_2$               &  2     & 1 & $2^{10}$       &  7     &  $  x \left( x^8+ \sum_{i=1}^7 a_i x^i   +1 \right) $ \\

2 & &  $ V_4$             &  2     & 2 & $2^7$       &  4     &  $ x^{10}+  \sum_{i=1}^4 a_i x^{2i}+1 $ \\ 

\textbf{\color{blue}\color{blue}3} & $C_m$ & $C_4$               &  2     & 2 & $2^4, 4^2$  &  3     &  $ x (x^8+a_3 x^6+a_2 x^4+a_1 x^2+1) $  \\

4 & & $C_6$               &  2     & 3 & $2^3, 3, 6$ &  2     &  $ x^9+a_1 x^3+a_2 x^6+1 $             \\

\textbf{\color{blue}\color{blue}5} &  & $C_3$              &  3     & 1 &$3^6$        &  3     &  $ x (x^4+a_1 x+a_2 x^2+a_3 x^3+1) $  \\

6 &  & $C_2 \times C_3$   &  3     & 2 & $2^2, 3^3 $  & 2     &  $ x^6+a_2 x^4+a_1 x^2+1 $             \\
\hline

7 &  & $D_6 \times C_3$   &  3     & 3 & $2^2, 3^2 $  & 1     &  $ x^6+a_1 x^3+1 $                   \\
8 & $D_{2m}$ & $V_4 \times C_3 $  &  3     & 2 & $2^2,  3, 6$ & 1     &  $ (x^2-1)(x^4+a_1 x^2+1) $         \\
9 & &  $V_4 \times C_3$   &  3     & 2 & $2^2, 3, 6$ & 1     &  $ x (x^4+a_1 x^2+1)   $              \\

\hline \hline
\end{tabular}

\label{g=4}
\end{table}

If $\bAut(\X)$ is isomorphic to a cyclic group, then there are two signatures for each of the groups $C_{18}$ and $C_{15}$.  In each case, both signatures give the same curve, namely $y^2 = x^9+1$ and $y^3=x^5+1$ respectively.  The left cases are given by Nrs. 1 to 6 in  Table~\ref{g=4}. As all cases, with the exception of Nrs. 1, 3 and 5, the curves are of odd signature; so definable over their field of moduli by Theorem \ref{teoAQ}.

\end{proof}

\s
\subsubsection{\bf Genus $5 \leq g \leq 10$}
We proceed to indicate which cases in the table provided in Section \ref{Sec:tablas} are definable over the field of moduli. 

\begin{enumerate}
\item Genus $5$: We may see from the table in Section \ref{Sec:tablas} that, for $g=5$, there are 20 cases to consider. By Theorem \ref{teounico} all cases, from Nr. 8 to Nr. 20, are definable over their field of moduli. The left cases, with the exception of Nrs. 1, 2 and 6, are of odd signature, so they are definable over their field of moduli.

\s
\item Genus $6$:
We see from the table in Section \ref{Sec:tablas} that, for $g=6$, there are 36 cases to consider. By Theorem \ref{teounico} all cases, from Nr. 17 to Nr. 36, are definable over their field of moduli. The left cases, with the exception of Nrs. 9, 10, 13 and 15, are of odd signature, so they are definable over their field of moduli.

\s
\item Genus $7$:
We see from the table in Section \ref{Sec:tablas} that, for $g=7$, there are 27 cases to consider. By Theorem \ref{teounico} all cases, from Nr. 14 to Nr. 27, are definable over their field of moduli. The left cases, with the exception of Nrs. 1, 2 and 11, are of odd signature, so they are definable over their field of moduli.

\s
\item Genus $8$:
We see from the table in Section \ref{Sec:tablas} that, for $g=8$, there are 22 cases to consider. By Theorem \ref{teounico} all cases, from Nr. 9 to Nr. 22, are definable over their field of moduli. The left cases, with the exception of Nrs. 2, 6, 7 and 8, are of odd signature, so they are definable over their field of moduli.

\s
\item Genus $9$:
We see from the table in Section \ref{Sec:tablas} that, for $g=9$, there are 50 cases to consider. By Theorem \ref{teounico} all cases, from Nr. 23 to Nr. 50, are definable over their field of moduli. The left cases, with the exception of Nrs. 1, 3, 4, 14, 16 and 20, are of odd signature, so they are definable over their field of moduli.

\s
\item Genus $10$:
We see from the last table in Section \ref{Sec:tablas} that, for $g=10$, there are 55 cases to consider. From them, there are 18 hyperelliptic, 18 trigonal, and 4 quintagonal. By Theorem \ref{teounico} all cases, from Nr. 24 to Nr. 55, are definable over their field of moduli. The left cases, with the exception of Nrs. 2, 3, 16, 17, 19, 20 and 23, are of odd signature, so they are definable over their field of moduli.

\end{enumerate}

\section{Tables of superelliptic curves of genus $5 \leq g \leq 10$}\label{Sec:tablas}

The following tables are taken from \cite[Table~1]{zhupa}.  The first column of the tables is simply a counter for each genus  $5 \leq g \leq 10$.   The second column is the reduced automorphism group and the third column  some information about  the full automorphism group. Notice that such column is left mostly empty, but a presentation of the group via its generators can be found in \cite{Sa} for all the cases. 
In the fourth column is the level $n$ of the superelliptic curve.  Hence, the equation of the curve is given by  $y^n = f(x)$, 
where $f(x)$ is the polynomial displayed in the last column.  Columns 5 and 6 respectively represent the order of an automorphism in the reduced automorphism group and the signature of the covering $\X \to \X/ G$. The seventh column represents the dimension of the corresponding locus in the moduli space $\M_g$. 

In \cite{zhupa}, the authors create a database of superelliptic curves.  Moreover, they display all curves of genus $g \leq 10$ in \cite[Table 1]{zhupa}.  We present such tables below. 
  The first column of the table represents the case from Table~1 of \cite{Sa}, the second column is the reduced automorphism group.  In the third column is the full automorphism group. Such groups are well known and  only  the 'obvious' cases are displayed, for full details one can check \cite{Sa} and \cite{Sa-sh}.  

In the fourth column is the level $n$ of the superelliptic curve; see \cite{super1}.  Hence, the equation of the curve is given by  $y^n = f(x)$, 
where $f(x)$ is the polynomial displayed in the last column.  Columns 5 and 6, respectively,  represent the order of the superelliptic  automorphism in the reduced automorphism group and the signature of the covering $\X \to \X/ G$. The sixth column represents the dimension of the corresponding locus in the moduli space $\M_g$. 
Throughout these tables $f_1 (x)$ is as follows
\[ f_1(x) = x^{12}-a_1 x^{10}-33x^8+2a_1x^6-33x^4-a_1x^2+1 \]
for  $a_1 \in \C$. 
In \cite[Table~1]{zhupa}  the signatures of the coverings are not fully given.  Indeed, for a full signature $\left( \sigma_1, \dots , \sigma_r\right)$, we know that   $\sigma_r = \sigma_{r-1}^{-1}\cdots \sigma_1^{-1}$.  Hence,  the  authors present only $(\sigma_1, \dots , \sigma_{r-1})$. In our Table~3 that follows, we present the 
full signature $\left( \sigma_1, \dots , \sigma_r\right)$.  

For example, in the case $g=5$, Nr.1,   we have a signature of eight branch points, each corresponding to a double transposition.  Such signature is presented as $2^8$.  
Another clue for the reader must be that    the moduli dimension is always $\delta= r-3$, where $r$ is the number of branch points of the covering.

\begin{rem}
It is worth mentioning, to avoid any confusion, that the equations of the curves in column eight are not over the field of moduli of the corresponding curve. They are used only to help identify the corresponding family. 
\end{rem}

\subsection*{Acknowledgements}
The authors want to thanks Jeroen Sijsling for his comments and suggestions to a previous version which help us to correctly state some points of this paper.

\begin{Small} 

\begin{table}[hbt]
\caption{Superelliptic curves for genus $5 \leq g \leq 10$}
\begin{tabular}{|l|l|l|l|l|l|l|l|}
\hline
Nr. & $\overline G$             & G&$n$  &$m$ & sig. & $\delta$ & Equation $y^n=f(x)$ \\
\hline \hline
\multicolumn{8}{c}{Genus 5} \\
\hline \hline
 \textbf{\color{blue}\color{blue}1}& \multirow{4}{*}{$C_m$} &   $V_4$            &  2  & 2  & $2^8$      & 5  & $x^{12} + \sum_{i=1}^5 a_i x^{2i} + 1$  \\ 
 \textbf{\color{blue}\color{blue}2}&                        &   $C_3 \times C_2$ &  2  & 3  & $2^4, 3^2$ & 3  & $x^{12} + \sum_{i=1}^3 a_i x^{3i} + 1$ \\   
 3&                        &   $C_2\times C_4$  &  2  & 4  & $2^3, 4^2$ & 2  & $x^{12} + a_2 \, x^8 + a_1\, x^4 +1$ \\  
 4&                        & $C_{22}$           &  2  & 11 & 2, 11, 22  & 0  & $x^{11}+1$ \\ 
 5&                        & $C_{22}$           &  11 & 2  & 2, 22, 22  & 0  & $x^2+1$ \\ 
 \textbf{\color{blue}\color{blue}6}&                        & $C_2$              &  2  & 1  & $2^{12}$   & 9  & $x (x^{10}+\sum_{i=1}^9 a_ix^i +1)$ \\ 
 7&                        & $C_4$              &  2  & 2  & $2^5$, $4^2$      & 4  & $x (x^{10}+\sum_{i=1}^4 a_i x^{2i} +1)$ \\ 
\hline
 8& \multirow{4}{*}{$D_{2m}$} &  &  2 & 2  & $2^6$        &  3 & $ \prod_{i=1}^3 (x^4+a_i x^2+1) $ \\ 
 9&                           &  &  2 & 3  & $2^4$, 3     &  2 & $(x^6+a_1x^3+1)(x^6+a_2x^3+1)$ \\   
 10&                           &  &  2 & 6  & $2^3$, 6     &  1 & $x^{12}+a_1x^6+1$ \\  
 11&                           &  &  2 & 4  & $2^2$, $4^2$ &  1 & $(x^4-1)(x^8+ a_1 x^4 +1)$  \\ 
 12&                           &  &  2 & 12 & 2, 4, 12     &  0 & $x^{12}-1$  \\ 
 13&                           &  &  2 & 5  & $2^3$, 10    &  1 & $x(x^{10}+a_1x^5+1)$ \\ 
 14&                           &  &  2 & 2  & $2^3$, $4^2$ &  2 & $(x^4-1)(x^4+a_1x^2+1)(x^4+a_2x^2+1)$ \\  
 15&                           &  &  2 & 3  & 2, 3, $4^2$  &  1 & $(x^6-1)(x^6+a_1x^3+1)$ \\  
 16&                           &  &  2 & 2  & $2^3$, $4^2$ &  2 & $x(x^2-1)(x^4+a_1x^2+1)(x^4+a_2x^2+1)$ \\  
 17&                           &  &  2 & 10 & 2, 4, 20     &  0 & $x (x^{10} -1)$ \\   
\hline   
18 &   $A_4$                     &  & 2 &  & $2^2$, $3^2$ & 1 & $f_1(x)$ \\  
\hline   
19 &   $S_4$                     &  & 2  & 0 & 3, $4^2$ & 0 & $x^{12} - 33x^8-33x^4+1$ \\  
\hline   
20 &   $A_5$                     &  & 2 &  & 2,3,10 & 0 & $x (x^{10} + 11x^5-1)$ \\  
\hline \hline
%
\end{tabular}
\end{table}

\addtocounter{table}{-1}
\begin{table}
\caption{(Cont.)}
\begin{tabular}{|l|l|l|l|l|l|l|l|}
\hline
Nr. & $\overline G$             & G&$n$  &$m$ & sig. & $\delta$ & Equation $y^n=f(x)$ \\
\hline  
\multicolumn{8}{c}{Genus 6} \\
\hline \hline
 1& \multirow{4}{*}{$C_m$} &   $V_4$        &  2  & 2      & $2^9$      & 6 & $x^{14} + \sum_{i=1}^6 a_i x^{2i} + 1$  \\
 2&                        &   $C_{26}$        &  2  & 13  & 2, 13, 26 & 0 & $x^{13} + 1$ \\
 3&                        &   $C_{21}$        &  3  & 7   & 3, 7, 21& 0  & $x^7 +1$ \\
 4&                        & $C_{20}$  &  4  & 5           & 4, 5, 20   & 0   & $x^{5}+1$ \\
 5&                        & $C_{10}$  &  5 & 2            &   $2, 5, 5, 10$  &  1    & $x^4+a_1x^2+1$ \\
 6&                        & $C_{20}$     &  5  & 4        &  4, 5, 20   & 0 & $x^4+1$ \\
 7&                        & $C_{21}$     &  7  & 3        & 3, 7 , 21   & 0 & $x^3 +1$ \\
 8&                        & $C_{26}$     &  13 & 2  & 2, 13, 26   & 0 & $x^2 +1$ \\
\textbf{\color{blue}\color{blue} 9}&                        & $C_{2}$     &  2  & 1  & $2^{14}$    & 11 & $x ( x^{12} +\sum_{i=1}^{11} a_i x^{i}+1 )$ \\
 \textbf{\color{blue}\color{blue}10}&                        & $C_{4}$     &  2  & 2  & $2^6, 4^2$    & 5 & $x (x^{12} +\sum_{i=1}^{5} a_i x^{2i}+1)$ \\
 11&                        & $C_{6}$     &  2  & 3  & $2^3, 3^2, 6^2$    & 3 & $x (x^{12} +\sum_{i=1}^{3} a_i x^{3i}+1)$ \\
 12&                        & $C_{8}$     &  2  & 4  & $2^3, 8^2$    & 2 & $x (x^{12}+ \sum_{i=1}^{2} a_i x^{4i}+1) $ \\
 \textbf{\color{blue}\color{blue}13}&                        & $C_{3}$     & 3  & 1  & $3^8$    & 5 & $x^{6}+ \sum_{i=1}^{5} a_i x^{i}+1$ \\
14&                        & $C_{6}$     &  3  & 2  & $3^3, 6^2$    & 2 & $x^{6}+a_2x^4+a_1x^2+1$ \\
 \textbf{\color{blue}\color{blue}15}&                        & $C_{4}$     &  4  & 1  & $4^6$    & 3 & $x^{4}+ \sum_{i=1}^{3} a_i x^{i}+1$ \\
 16&                        & $C_{5}$     &  5  & 1  & $5^5$    & 2 & $x^3 +a_1x+a_2x^2+1$ \\
\hline
 17& \multirow{4}{*}{$D_{2m}$} &$ D_{14}\times C_2$ &  2 & 7  & $2^3, 7$        &  1 & $x^{14}+a_1x^7+1)$ \\
 18&                           & $G_5$ &  2 & 2  & $2^5$, 4     &  3 & $(x^2-1) \prod_{i=1}^3 (x^4+a_i x^2+1) $ \\
 19&                           & $G_5$ &  2 & 14  & $2, 4, 14$     &  0 & $x^{14}- 1$ \\
 20&                           & $ D_{10}\times C_2$ &  5 & 5  & $2, 5, 10$ &  0 & $x^5-1$  \\
 21&                           & $D_8$ &  2 & 2 & $2^5, 4 $     & 3 & $x \cdot  \prod_{i=1}^3 (x^4+a_i x^2+1)  $  \\
 22&                           &$ D_{6}\times C_2$  &  2 & 3  & $2^4$, 6    &  2 & $x \cdot \prod_{i=1}^2 (x^4+a_i x^2+1)$ \\
  23&   \multirow{4}{*}{$D_{2m}$}                        &$D_{24}$  &  2 & 6  & $2^3$, 12 &  1 & $x(x^{12}+a_1x^6+1)$ \\
 24&                           & $ D_{6}\times C_3$ &  3 & 3  &  $2^2, 3, 9$  &  1 & $x(x^6+a_1x^3+1)$ \\
 25&                           & $ D_{16}$ &  4 & 2  & $2^2$, 4, 8 &  1 & $x(x^4+a_1x^2+1)$ \\
 26&                           & $G_8$ &  2 & 4 & $2^2$, 4, 8     &  1 & $x (x^{4} -1)(x^8+a_1x^4+1)$ \\
 27&                           & $G_8$ &  2 & 12 & 2, 4, 24     &  0 & $x (x^{12} -1)$ \\
 28&                           &$ V_{4}\times C_3$  &  3 & 2 & 2, 3, $6^2$     &  1 & $x (x^{2} -1)(x^4+a_1x^2+1)$ \\
 29&                           & $ D_{12}\times C_3$ &  3 & 6 & 2, 6, 18     &  0 & $x (x^{6} -1)$ \\
 30&                           &$G_8$  &  4 & 4 & 2, 8, 16     &  0 & $x (x^{4} -1)$ \\
 31&                           & $ D_{6}\times C_5$ &  5 &3 & 2, 10, 15     &  0 & $x (x^{3} -1)$ \\
 32&                           &$ V_{4}\times C_7$  &  7 & 2 & 2, $14^2$     &  0 & $x (x^{2} -1)$ \\
 33&                           &$G_9$  &  2 & 2 & $2^2$, $4^3$     &  2 & $x (x^{4} -1) \cdot \prod_{i=1}^2 (x^4+a_i x^2+1)$ \\
 34&                           & $G_9$ &  2 & 3 & $2, 4^2, 6$     &  1 & $x (x^{6} -1)(x^6+a_1x^3+1)$ \\
\hline
 35&   $S_4$                     & $ G_{18}$ & 4  & 0 & 2, 3, 16 & 0 & $x(x^4-1)$ \\
 36&                             & $ G_{19}$ & 2  & 0 & 2, 6, 8 & 0 & $x(x^4-1)(x^8+14x^4+1)$ \\
\hline \hline
\multicolumn{8}{c}{Genus 7} \\
\hline \hline
 \textbf{\color{blue}\color{blue}1}& \multirow{4}{*}{$C_m$} &   $V_4$        &  2  & 2             & $2^{10}$      & 7 & $x^{16} + \sum_{i=1}^7 a_i x^{2i} + 1$  \\
 \textbf{\color{blue}\color{blue}2}&                        &   $C_2 \times C_4$        &  2  & 4  & $2^4, 4^2$ & 3 & $x^{16} + \sum_{i=1}^3 a_i x^{4i} + 1$ \\
 3&                        &   ${C_3}^2$        &  3  & 3         & $3^5$ & 2  & $x^{9} + a_2 x^6 + a_1 x^3 +1$ \\
 4&                        & $C_{6}$  &  2  & 3                   &  $2^5$, 3, 6    & 4   & $x^{15}+ \sum_{i=1}^4 a_1 x^{3i} +1$ \\
 5&                        & $C_{10}$  &  2 & 5                   &   $2^3$, 5, 10   &  2    & $x^{15}+a_1x^5+a_2x^{10}+1$ \\
 6&                        & $C_{30}$     &  2  & 15              &  2, 15, 30   & 0 & $x^{15} +1$ \\
 7&                        & $C_6$     &  3  & 2                  & 2, $3^4$, 6         &   3 & $x^8+a_3x^6+a_2x^4+a_1x^2 +1$ \\
 8&                        & $C_{12}$     &  3  & 4               & $3^2$, 4, 12         &   1 & $x^8+a_1x^4+1$ \\
 9&                        & $C_{24}$     &  3  & 8               & 3, 8, 24         &   0 & $x^8+1$ \\
 10&                        & $C_{30}$     &  15  & 2             & 2, 15, 30      &   0 & $x^2+1$ \\
 \textbf{\color{blue}\color{blue}11}&                        & $C_{2}$     &  2  & 1  & $2^{16}$         &   13 & $x(x^{14}+\sum_{i=1}^{13} a_i x^{i}+1)$ \\
 12&                        & $C_{4}$     &  2  & 2  & $2^7, 4^2$       &   6 & $x (x^{14}+\sum_{i=1}^6 a_i x^{2i}+1)$ \\
13&                        & $C_{3}$     &  3  & 1  & $3^9$       &   6 & $x^{7}+\sum_{i=1}^6 a_i x^{i}+1$ \\
\hline
 14& \multirow{1}{*}{$D_{2m}$} & $V_4\times C_2$ &  2 & 2  & $2^7$        &  4 & $\prod_{i=1}^4(x^4+a_ix^2+1)$ \\
 15&                           &$D_8\times C_2$  &  2 & 4  & $2^4$, 4     &  2 & $(x^8+a_1x^4+1)(x^8+a_2x^4+1)$ \\
\hline  
\end{tabular}
\end{table}

\addtocounter{table}{-1}
\begin{table}
\caption{(Cont.)}
\begin{tabular}{|l|l|l|l|l|l|l|l|}
\hline
Nr. & $\overline G$             & G&$n$  &$m$ & sig. & $\delta$ & Equation $y^n=f(x)$ \\
\hline \hline
 16&   \multirow{4}{*}{$D_{2m}$}                        &$D_{16}\times C_2$  &  2 & 8  & $2^3$, 8  &  1 & $x^{16}+a_1x^8+1$ \\
 17&                           &$G_5$  &  2 & 16  & 2, 4, 16 &  0 & $x^{16}-1$  \\
 18&                           & $D_6\times C_3$ &  3 & 3 & 2, $3^2$, 6    &  1 & $(x^3 - 1)(x^6+a_1x^3+1)$  \\
 19&                           & $D_{18}\times C_3$ &  3 & 9 & 2, 6, 9  &  0 & $x^9-1$  \\
 20&                           & $D_{14}\times C_2$ &  2 & 7  & $2^3$, 14    &  1 & $x(x^{14}+a_1x^7+1)$ \\
 21&                           &$G_7$  &  2 & 2  & $2^4$, $4^2$ &  3 & $(x^4-1) \, \prod_{i=1}^3 (x^4+a_i x^2+1) $ \\
 22&                           & $G_7$ &  2 & 4  & 2, $4^3$  &  1 & $(x^8-1)(x^8+a_1x^4+1)$ \\
 23&                           & $G_8$ &  2 & 2  & $2^4$, $4^2$ &  3 & $x(x^2-1)\, \prod_{i=1}^3 (x^4+a_i x^2+1) $ \\
 24&                           & $G_8$ &  2 & 14 & 2, 4, 28     &  0 & $x (x^{14} -1)$ \\
 25&                           & $D_{14}\times C_3$ &  3 & 7 & 2, 6, 21     &  0 & $x (x^{7} -1)$ \\
 26&                           & $G_8$ &  8 & 2 & 2,${16}^2$     &  0 & $x (x^{2} -1)$ \\
\hline 
27 &   $A_4$                     &$K$  & 2 & 0  & $2^2$, 3, 6 & 1 & $(x^4+2 \sqrt{-3} x^2+1) \, f_1(x)$ \\
\hline \hline
\multicolumn{8}{c}{Genus 8} \\
\hline \hline
1& \multirow{4}{*}{$C_m$} &   $V_4$        &  2  & 2             & $2^{11}$      & 8 & $x^{18} + \sum_{i=1}^8 a_i x^{2i} + 1$  \\
 \textbf{\color{blue}\color{blue}2}&                        &   $C_2 \times C_3$        &  2  & 3  &    $2^6, 3^2$ & 5 & $x^{18} + \sum_{i=1}^5 a_i x^{3i} + 1$ \\
3&                        &   $C_2\times C_6$        &  2  & 6   & $2^3, 6^2$ & 2  & $x^{18} + a_1 x^6 + a_2 x^{12} +1$ \\
 4&                        & $C_{34}$  &  2  & 17                 &   2, 17, 34   & 0   & $x^{17}+1$ \\
 5&                        & $C_{34}$  &  17 & 2                  &   2, 17 , 34  &  0    & $x^2+1$ \\
 \textbf{\color{blue}\color{blue}6}&                        & $C_2$     &  2  & 1  &  $2^{18}$   & 15 & $x(x^{16}+\sum_{i=1}^15 a_ix^i +1)$ \\
 \textbf{\color{blue}\color{blue}7}&                        & $C_4$     &  2  & 2  & $2^8, 4^2$        &  7 & $x(x^{16}+\sum_{i=1}^7 a_i x^{2i} +1)$ \\
  \textbf{\color{blue}\color{blue}8}&                        & $C_8$     &  2  & 4  & $2^4, 8^2$        & 3 & $x(x^{16}+a_1x^4+a_2x^8+a_3x^{12} +1)$ \\
\hline
 9& \multirow{4}{*}{$D_{2m}$} &$D_6\times C_2$  &  2 & 3  & $2^5$, 3    & 3 & $ \prod_{i=1}^3 (x^6+a_i x^3+1) $ \\
 10&                           &$D_{18}\times C_2$  &  2 & 9  & $2^3$, 9 & 1 & $x^{18}+a_1x^9+1$ \\
 11&                           & $G_5$ &  2 & 2  & $2^6$, 4     &  4 & $(x^2-1)\prod_{i=1}^4(x^4+a_ix^2+1)$ \\
 12&                           &$G_5$  &  2 & 6  & $2^2$, 4, 6 &  1 & $(x^6-1)(x^{12}+a_1x^6+1)$  \\
 13&                           &$G_5$  &  2 & 18 & 2, 4, 18     &  0 & $x^{18}-1$  \\
 14&                           &$D_8$  &  2 & 2  & $2^6$, 4    &  4 & $x\prod_{i=1}^4(x^4+a_ix^2+1)$ \\
 15&                           &$D_{16}$  &  2 & 4  & $2^4$, 8    &  2 & $x(x^8+a_1x^4+1)(x^8+a_2x^4+1)$ \\
 16&                           &$D_{32}$  &  2 & 8  & $2^3$, 16    &  1 & $x(x^{16}+a_1x^8+1)$ \\
 17&                           &$G_9$  &  2 & 3  & $2^2$, 3,  $4^2$ &  2 & $(x^6-1)(x^6+a_1x^3+1)(x^6+a_2x^3+1)$ \\
 18&                           &$G_8$  &  2 & 16  & 2, 4, 32 &  0 & $x(x^{16}-1)$ \\
 19&                           & $G_9$ &  2 & 2 & $2^3, 4^3$    &  3 & $x \, \prod_{i+1}^3 (x^6+a_i x^3+1) $ \\   
  20&                           & $G_9$ &  2 & 4 & 2, $ 4^2$, 8   &  1 & $x(x^8-1)(x^8+a_1x^4+1)$ \\
\hline
21 &   $A_4$                     & $K$ & 2 & 0 & 2, $3^2$, 4 & 1 & $x(x^4-1)\, f_1 (x)$ \\
\hline
22 &   $S_4$                     &$G_{22}$  & 2  & 0 & 3, 4, 8 & 0 & $x(x^4-1)(x^{12}-33x^8-33x^4+1)$ \\
\hline \hline
\multicolumn{8}{c}{Genus 9} \\
\hline \hline
 \textbf{\color{blue}\color{blue}1}& \multirow{4}{*}{$C_m$} &   $V_4$        &  2  & 2  & $2^{12}$      & 9 & $x^{20} + \sum_{i=1}^9 a_i x^{2i} + 1$  \\
 2&                        &   $C_2 \times C_4$ &  2  & 4  & $2^5, 4^2$ & 4 & $x^{20} + \sum_{i=1}^4 a_i x^{4i} + 1$ \\
 \textbf{\color{blue}\color{blue}3}&                        &   $C_2\times C_5$  &  2  & 5 & $2^4, 5^2$ & 3  & $x^{20} + a_1 x^5 + a_2 x^{10}+a_3x^{15} +1$ \\
 \textbf{\color{blue}\color{blue}4}&                        &   $C_2\times C_4$  &  4  & 2 & $2^2, 4^4$ & 3  & $x^{8} + a_1 x^2 + a_2 x^{4}+a_3x^{6} +1$ \\
 5&                        & $C_{38}$  &  2  & 19 &  2, 19, 38   & 0   & $x^{19}+1$ \\
 
 6&                        & $C_{6}$  &  3 & 2  &   2, $3^5$, 6   &  4    & $x^{10} + a_1 x^2 + a_2 x^{4}+a_3x^{6}+a_4x^8 +1$ \\
 7&                        & $C_{15}$  &  3 & 5  &   $3^2$, 5, 15  &  1   & $x^{10} + a_1 x^5 +1$ \\
 8&                        & $C_{30}$  &  3 & 10  &  3, $10^2$  &  0    & $x^{10} +1$ \\
 9&                        & $C_{28}$  &  4 & 7  &  4, $7^2$   &  0    & $x^7 +1$ \\
 10&                        & $C_{14}$  &  7 & 2  &   $2, 7^2, 14$   &  1   & $x^4+a_1x^2+1$ \\
  11&                        & $C_{28}$  &  7 & 4  &  $4^2$, 7   & 0    & $x^4 +1$ \\
  12&                        & $C_{30}$  &  10 & 3  &  $3^2$, 10    &  0      & $x^3 +1$ \\
  13&                        & $C_{38}$  &  19 & 2  &   $2^2$, 19   &  0      & $x^2  +1$ \\
   \textbf{\color{blue}\color{blue}14}&                        & $C_2$     &  2  & 1  &  $2^{20}$   & 17 & $x (x^{18}+\sum_{i=1}^{17} a_ix^i +1)$ \\
 15&                        & $C_4$     &  2  & 2  & $2^9, 4^2$        &   8 & $x (x^{18}+\sum_{i=1}^8 a_i x^{2i} +1) $ \\
  \textbf{\color{blue}\color{blue}16}&                        & $C_6$     &  2  & 3  & $2^6, 6^2$        &   5 & $x (x^{18}+\sum_{i=1}^5 a_i x^{3i} +1)$ \\
\hline  
\end{tabular}
\end{table}

\addtocounter{table}{-1}
\begin{table}
\caption{(Cont.)}
\begin{tabular}{|l|l|l|l|l|l|l|l|}
\hline 
\hline
Nr. & $\overline G$             & G&$n$  &$m$ & sig. & $\delta$ & Equation $y^n=f(x)$ \\
\hline 
 17&                        & $C_{12}$     &  2  & 6  & $2^3, {12}^2$  &   2 & $ x (x^{18}+ a_1x^6+a_2x^{12} +1) $ \\
 18&                        & $C_{3}$     &  3   & 1  & $3^{11}$  &   8 & $x^{9}+ \sum_{i=1}^8 a_i x^{i}+1$ \\

19&                        & $C_{9}$     &  3   & 3  & $3^3, 9^2$  &   2 & $x^{9}+ a_2x^6+a_1x^3+1$ \\
 \textbf{\color{blue}\color{blue}20}&                        & $C_{4}$     &  4   & 1  & $4^8$  &   5 & $x^{6}+ \sum_{i=1}^5 a_i x^{i}+1$ \\
 21&                        & $C_{8}$     &  4   & 2  & $4^3, 8^2$  &   2 & $x^6+a_2x^4+a_1x^2+1$ \\
 22&                        & $C_{7}$     &  7   & 1  & $7^5$  &   2 & $x^3+a_1x+a_2x^2+1$ \\
\hline
 23& \multirow{4}{*}{$D_{2m}$} & $V_4\times C_2$ &  2 & 2  & $2^8$     &  5 & $\prod_{i=1}^5(x^4+a_ix^2+1))$ \\
 24&                           & $ D_{10}\times C_2$  &  2 & 5  & $2^4$, 5     &  2 & $(x^{10}+a_1x^5+1)(x^{10}+a_2x^5+1)$ \\
 25&                           &$ D_{20}\times C_2$  &  2 & 10  & $2^3$, 10     &  1 & $x^{20}+a_1x^{10}+1$ \\
 26&                           &$ V_{4}\times C_4$  &  4 & 2  & $2^3, 4^2$      &  2 & $(x^4+a_1x^2+1)(x^4+a_2x^2+1)$ \\
 27&                           &$ D_{8}\times C_4$  &  4 & 4  & $2^2, 4^2$      &  1 & $x^8+a_1x^4+1$ \\
 28&                           &$G_5$  &  2 & 4  & $2^3$, $4^2$ &  2 & $(x^4-1)(x^8+a_1x^4+1)(x^8+a_2x^4+1)$  \\
 29&                           &$G_5$  &  2 & 20 & 2, 4, 20     &  0 & $x^{20}-1$  \\
 30&                           &$G_5$  &  4 & 8 & 2, $8^2$     &  0 & $x^{8}-1$  \\
 31&                           &$D_6\times C_2$  &  2 & 3  & $2^5$, 6    &  3 & $x \, \prod_{i=1}^3 (x^6+a_i x^3+1)$ \\
 32&                           &$D_{18}\times C_2$  &  2 & 9  & $2^3$, 18    &  1& $x(x^{18}+a_1x^9+1)$ \\
 33&                           &$D_{6}\times C_4$  &  4 & 3  & $2^2$, 4, 12    &  1& $x(x^6+a_1x^3+1)$ \\
 34&                           & $G_7$ &  2 & 2  & $2^5$, $4^2$ & 4 & $(x^4-1)\prod_{i=1}^4(x^4+a_ix^2+1)$ \\
 35&                           &$G_9$  &  2 & 5  & 2, $4^2$, 5 &  1 & $(x^{10}-1)(x^{10}+a_1x^5+1)$ \\
 36&                           & $G_7$ &  4 & 2  & 2, 4, $8^2$ & 1 & $(x^4-1)(x^4+a_1x^2+1)$ \\
 37&                           &$G_8$  &  2 & 2  & $2^5$, $4^2$ & 4 & $x(x^2-1)\prod_{i=1}^4(x^4+a_ix^2+1)$ \\
 38&                           & $G_8$ &  2 & 6 & $2^2$, 4, 12     &  1 & $x (x^{6} -1)(x^{12}+a_1x^6+1)$ \\
 39&                           & $G_8$ &  2 & 18 & 2, 4, 36     &  0 & $x (x^{18} -1)$ \\
 40&                           & $D_6\times C_3$ &  3 & 3 & 2, 3, 6, 9     &  1 & $x (x^{3} -1)(x^{6}+a_1x^3+1)$ \\
 41&                           & $D_{18}\times C_3$ &  3 & 9 & 2, 6, 27     & 0 & $x (x^{9}-1)$ \\
  42&                           & $G_8$ &  4 & 2 & 2, 4, $8^2$     &  1 & $x (x^{2} -1)(x^4+a_1x^2+1)$ \\
   43&                           & $G_8$ &  4 & 6 & 2, 8, 24     &  0 & $x (x^{6} -1)$ \\
  44&                           & $D_6\times C_7$ &  7 & 3 & 2, 14, 21     &  0 & $x (x^{3} -1)$ \\
   45&                           & $G_8$ &  10 & 2 & 2, $20^2$     &  0 & $x (x^{2} -1)$ \\
  46&                           & $G_9$ &  2 & 3 & $2^2, 4^2, 6$     &  2 & $x (x^{6} -1)(x^6+a_1x^3+1)(x^6+a_2x^3+1)$ \\
\hline
47 &   $A_4$                     & $K$ & 2 & 0 & $2^2$, $6^2$ & 1 & $(x^8+14x^4+1) \, f_1(x) $ \\
\hline
 48&   $S_4$                     &$G_{17}$  & 4  & 0 & 2, 4, 12 & 0 & $x^{8}+14x^4+1$ \\
 49&                             &$G_{21}$  & 2  & 0 & $4^2$, 6 & 0 & $(x^{8}+14x^4+1)(x^{12}-33x^8-33x^4+1)$ \\
\hline
50 &   $A_5$                     &  & 2 &  & 2, 5, 6 & 0 & $x^{20}-228x^{15}+494x^{10}+228x^5+1$ \\
\hline \hline
\multicolumn{8}{c}{Genus 10} \\
\hline \hline
 1& \multirow{4}{*}{$C_m$} &   $V_4$        &  2  & 2  & $2^{13}$      & 10 & $x^{22} + \sum_{i=1}^{10} a_i x^{2i} + 1$  \\
 \textbf{\color{blue}\color{blue}2}&                        &   $C_2 \times C_3$ &  3  & 2 & $2^2, 3^6$ & 5 & $x^{12} + \sum_{i=1}^5 a_i x^{2i} + 1$ \\
 \textbf{\color{blue}\color{blue}3}&                        &   $C_3^2$  &  3  & 3 & $3^6$ & 3  & $x^{12} + a_1 x^3 + a_2 x^{6}+a_3x^{9} +1$ \\
 4&                        &   $C_3\times C_4$  &  3  & 4 & $3^3, 4^2$ & 2  & $x^{12} + a_1 x^4 + a_2 x^{8} +1$ \\
 5&                        &   $C_2\times C_6$  &  6  & 2 & $2^2, 6^3$ & 2  & $x^{6} + a_1 x^2 + a_2 x^{4} +1$ \\
 6&                        & $C_{6}$  &  2  & 3 &   $2^7$, 3, 6   & 6   & $x^{21}+ \sum_{i=1}^6 a_i x^{3i} +1$ \\
 7&                        & $C_{14}$  &  2 & 7  &   $2^3$, 7, 14  &  2   & $x^{21} + a_1 x^7+a_2x^{14} +1$ \\
 8&                        & $C_{42}$  & 2 & 21  &  2, 4, 21  &  0    & $x^{21} +1$ \\
  9&                        & $C_{33}$  &  3 & 11  &  3, $11^2$   &  0    & $x^{11} +1$ \\
 10&                        & $C_{10}$  &  5 & 2  &   2, $5^3$, 10  &  2   & $x^6+a_2x^4+a_1x^2+1$ \\
 11&                        & $C_{15}$  &  5 & 3  &  3, $5^2$, 15   & 1    & $x^6+a_1x^3 +1$ \\
 12&                        & $C_{30}$  &  5 & 6  &  5, $6^2$    &  0      & $x^6 +1$ \\
 13&                        & $C_{30}$  &  6 & 5  &   $5^{2}$, 6   &  0      & $x^5  +1$ \\
 14&                        & $C_{33}$     &  11  & 3  &  $3^2$, 11  & 0 & $x^3 +1$ \\
 15&                        & $C_{42}$     &  21  & 2  & 2, 21, 42       &   0 & $x^2+1$ \\
 \textbf{\color{blue}\color{blue}16}&                        & $C_2$     &  2  & 1  & $2^{22}$        &   19 & $x(x^{20}+\sum_{i=1}^{19} a_i x^{i} +1)$ \\
 \textbf{\color{blue}\color{blue}17}&                        & $C_4$     &  2  & 2  & $2^{10}, 4^2$        &   9 & $   x( x^{20}+\sum_{i=1}^9 a_i x^{2i} +1)$ \\
 18&                        & $C_8$     &  2  & 4  & $2^5, 8^2$        &   4 & $ x(x^{20}+a_1x^4+a_2x^8+a_3x^{12}+a_4x^{16} +1)$ \\ 
 \textbf{\color{blue}\color{blue}19}&                        & $C_{10}$     &  2  & 5  & $2^4, {10}^2$  &   3 & $x (x^{20}+ a_1x^5+a_2x^{10}+a_3x^{15} +1)$ \\
 \textbf{\color{blue}\color{blue}20}&                        & $C_{3}$     &  3   & 1  & $3^{12}$  &   9 & $x^{10}+ \sum_{i=1}^9 a_i x^{i}+1$ \\
\hline
\end{tabular}
\end{table}

\addtocounter{table}{-1}
\begin{table}
\caption{(Cont.)}
\begin{tabular}{|l|l|l|l|l|l|l|l|}
\hline
Nr. & $\overline G$             & G&$n$  &$m$ & sig. & $\delta$ & Equation $y^n=f(x)$ \\
\hline 
\hline
 21&                        & $C_{6}$     &  3   & 2  & $3^5, 6^2$  &   4 & $x^{10}+ a_1x^2+a_2x^4+a_3x^6+a_4x^8+1$ \\
 22&                        & $C_{5}$     &  5   & 1  & $5^7$  &   4 & $x^{5}+ \sum_{i=1}^4 a_i x^{i}+1$ \\
 \textbf{\color{blue}\color{blue}23}&                        & $C_{6}$     & 6   & 1  & $6^6$  &   3 & $x^4+a_1x+a_2x^2a_3x^3+1$ \\
\hline 
 24& \multirow{4}{*}{$D_{2m}$} & $D_{22}\times C_2$ &  2 & 11  & $2^3$, 11     &  1 & $x^{22}+a_1x^{11}+1$ \\
 25&                           & $ V_{4}\times C_3$  &  3 & 2  & $2^3, 3^3$     &  3 & $\prod_{i=1}^3 (x^4+a_i x^2+1) $ \\
 26&                           &$ D_{6}\times C_3$  &  3 & 3  & $2^2, 3^3$,    &  2 & $(x^6+a_1x^3+1)(x^6+a_2x^3+1)$ \\
 27&                           &$ D_{12}\times C_3$  &  3 & 6  & $2^2, 3, 6$      &  1 & $(x^{12}+a_1x^6+1$ \\
 28&                           &$ D_{6}\times C_6$  &  6 & 3  & $2^2, 3, 6$      &  1 & $x^6+a_1x^3+1$ \\
 29&                           &$G_5$  &  2 & 2  & $2^7$, 4 &  5 & $(x^2-1)\prod_{i=1}^5(x^4+a_ix^2+1)$  \\
 30&                           &$G_5$  &  2 & 22 & 2, 4, 22     &  0 & $x^{22}-1$  \\
 31&                           &$D_8\times C_3$  &  3 & 4 & 2, 3, 4, 6     &  1 & $(x^4-1)(x^8+a_1x^4+1)$  \\
 32&                           &$D_{24}\times C_3$  &  3 & 12  & 2, 6, 12   &  0& $x^{12}-1$ \\
 33&                           &$G_5$  &  6 & 2 & $2^2$, 6, 12     &  1 & $(x^{2}-1)(x^4+a_1x^2+1)$  \\
 34&                           &$G_5$  &  6 & 6 &    2, 6, 12     &  0 & $x^6-1$  \\
  35&                           &$D_8$  &  2 & 2 & $2^7$, 4     &  5 & $x\prod_{i=1}^5(x^4+a_ix^2+1)$  \\
 36&                           &$D_{10}\times C_2$  &  2 & 5  & $2^4$, 10    &  2& $x(x^{10}+a_1x^5+1)(x^{10}+a_2x^5+1)$ \\
 37&                           &$D_{40}$  &  2 & 10  & $2^3$, 20    &  1 & $x(x^{20}+a_1x^{10}+1)$ \\
 38&                           &$D_{10}\times C_3$  &  3 & 5  & $2^2$, 3, 15    &  1& $x(x^{10}+a_1x^5+1)$ \\
 39&                           &$D_{24}$  &  6 & 2  & $2^2$, 6, 12    &  1 & $x(x^{4}+a_1x^{2}+1)$ \\
 40&                           & $V_4\times C_3$ &  3 & 2  & 2, $3^2$, $6^2$ & 2 & $(x^2-1)(x^4+a_1x^2+1)(x^4+a_2x^2+1)$ \\
 41&                           &$D_6\times C_3$  &  3 & 3  &  $3^2, 6^2$ &  1 & $(x^{6}-1)(x^{6}+a_1x^3+1)$ \\
 42&                           &$G_8$  &  2 & 4  & $2^3$, 4, 8 & 2 & $x(x^4-1)(x^8+a_1x^4+1)(x^8+a_2x^4+1)$ \\
 43&                           & $G_8$ &  2 & 20 & 2, 4, 40    &  0 & $x (x^{20} -1)$ \\
 44&                           & $V_4\times C_3$ &  3 & 2 & 2,$3^2, 6^2$   &  2& $x(x^2-1)(x^4+a_1x^2+1)(x^4+a_2x^2+1)$ \\
 45&                           & $D_{20}\times C_3$ &  3 & 10 & 2, 6, 30   & 0 & $x (x^{10} -1)$ \\
 46&                           & $D_{10}\times C_5$ &  5 & 5 & 2, 10, 25   & 0 & $x (x^{5}-1)$ \\
 47&                           & $G_8$ &  6 & 4 & 2, 12, 24     &  0 & $x (x^{4} -1)$ \\
 48&                           & $V_4\times C_{11}$ &  11 & 2 & 2, $22^2$     &  0 & $x (x^{2} -1)$ \\
 49&                           & $G_9$ &  2 & 2 & $2^4, 4^3$     & 4 & $x(x^4-1)\prod_{i=1}^4(x^4+a_ix^2+1)$ \\
 50&                           & $G_9$ &  2 & 5 & $2, 4^2, 10$     & 1 & $x(x^{10}-1)(x^{10}+a_1x^5+1)$ \\
\hline
51 &   $A_4$                     & & 3 & 0 & 2, $3^3$ & 1 & $ f_1 (x) $ \\
52&                              & & 2 & 0 & 2, 3, 4, 6 & 1 & $x(x^4-1)(x^4+ 2 \sqrt{-3} \, x^2+1)\, f_1(x)$\\
\hline
 53&   $S_4$                     &$G_{18}$  & 6  & 0 & 2, 3, 24 & 0 & $x(x^4-1)$ \\
 54&                             &$S_4\times C_3$  & 3  & 0 & 3, 4, 6 & 0 & $x^{12}-33x^8-33x^4+1$ \\
\hline
55 &   $A_5$                     &$A_5\times C_3$  & 3 & 0 & 2, 3, 15 & 0 & $x(x^{10}+11x^5-1)$ \\
\hline
\end{tabular}
\end{table}
 
\end{Small} 



\nocite{*}

\bibliographystyle{plainurl}

\bibliography{ref}{}


\end{document}